\documentclass{amsart}
\usepackage{amsfonts,amssymb,amscd,amsmath,enumerate,verbatim,calc}

\newcommand{\CM}{Cohen-Macaulay}

\newcommand{\bx}{\mathbf{x}}

\newcommand{\n}{\mathfrak{n} }
\newcommand{\m}{\mathfrak{m} }

\newcommand{\Z}{\mathbb{Z} }
\newcommand{\Xb}{\mathbf{X}^\bullet}

\newcommand{\Yb}{\mathbf{Y}^\bullet}

\newcommand{\Pb}{\mathbf{P}^\bullet}

\newcommand{\rt}{\rightarrow}

\newcommand{\ov}{\overline}

\newcommand{\codim}{\operatorname{codim}}
\newcommand{\embdim}{\operatorname{embbdim}}

\newcommand{\depth}{\operatorname{depth}}
\newcommand{\projdim}{\operatorname{projdim}}

\newcommand{\proj}{\operatorname{proj}}

\newcommand{\type}{\operatorname{type}}

\newcommand{\Tor}{\operatorname{Tor}}
\newcommand{\rank}{\operatorname{rank}}

\newcommand{\gr}{\operatorname{gr}}
\newcommand{\cx}{\operatorname{cx}}
\newcommand{\curv}{\operatorname{curv}}

\newcommand{\Hom}{\operatorname{Hom}}

\theoremstyle{plain}

\newtheorem{theorem}{Theorem}[section]

\newtheorem{lemma}[theorem]{Lemma}
\newtheorem{proposition}[theorem]{Proposition}

\theoremstyle{definition}

\newtheorem{s}[theorem]{}
\newtheorem{remark}[theorem]{Remark}
\newtheorem{example}[theorem]{Example}

\theoremstyle{remark}
\begin{document}

\title[Non-extremal]{Bounds on multiplicity of MCM modules having non-extremal growth of betti-numbers}
\author{Tony~J.~Puthenpurakal}
\date{\today}
\address{Department of Mathematics, IIT Bombay, Powai, Mumbai 400 076, India}

\email{tputhen@gmail.com}

 \begin{abstract}
 Let $(A,\mathfrak{m})$ be a Cohen-Macaulay local ring of dimension $d$ and residue field $k$.
  Let $M$ be a maximal \CM \ $A$-module. Let  $e(M)$ be the multiplicity of $M$ and let $\mu(M)$ denote the number of its minimal generators.
\begin{enumerate}
  \item Assume $A$ is not a complete intersection. If $\text{curv}(M) < \text{curv}(k)$ then we prove that under mild conditions, $e(M) \geq \mu(M)(1 + \curv(k))$.
  \item Assume $A$ is a complete intersection. If $\text{cx}(M) < \text{cx}(k)$ then we prove that $e(M) \geq 2\mu(M)$.
\end{enumerate}
In both cases we give examples which shows our results are sharp.
\end{abstract}
 \maketitle
\section{introduction}
In this paper all rings considered are Noetherian and local. Furthermore all modules considered will be finitely generated.
Let $(A,\m)$ be a  local ring with residue field $k$ and let $E$ be  an $A$-module. By $\ell(-)$ we denote the length function on $A$-modules. Let $\beta_n^A(E) = \ell(\Tor^A_n(E,k))$ for $n \geq 0$ be the $n^{th}$-betti number of $E$.
We set the \emph{complexity} of $E$ to be
\[
\cx_A(E) = \min\{ m \in \mathbb{N} \mid \limsup_{n \rt \infty} \beta_n(E)/n^{m-1} < \infty \}.
\]
It is possible that $\cx_A(E)$ is infinite. However if $A$ is a complete intersection of codimension $c$ then $\cx_A(E) \leq \cx_A(k) = c$.
We set the \emph{curvature} of $E$ to be
\[
\curv_A(E) = \limsup_{n \rt \infty}\sqrt[n]{\beta_n^A(E)}.
\]
We drop the subscript $A$ if it is clear from the context.
It can be shown that $\curv(E) \leq \curv(k) $, see  \cite[4.1.9]{A} and $\curv(k) < \infty$; see \cite[4.1.5]{A}.  It can also be shown that $A$ is \emph{not} a complete intersection if and only if $\curv(k) > 1$; see \cite[8.2.1]{A}.
An $A$-module $E$ is said to be \emph{extremal} if $\curv(E) = \curv(k)$ and $\cx(E) = \cx(k)$.

If $\dim E = r$ then $e(E) = \lim_{n \rt \infty}r!\ell(E/\m^n E)/n^r$ is called the multiplicity of $E$. We also let $\mu(E) = \ell(E/\m E)$ denote the number of minimal generators of $E$.
It is not difficult to prove that if $E$ is \CM \ then $e(E) \geq \mu(E)$.

We consider a maximal \CM \ (= MCM) $A$-module $M$. We are interested in the fraction $\xi(M) = e(M)/\mu(M) \geq 1$. Recall an MCM $A$-module $U$ is said to be Ulrich if $e(U) = \mu(U)$. It is not difficult to prove Ulrich modules are extremal. We prove:
\begin{theorem}
\label{main}
Let $(A,\m)$ be a Cohen-Macaulay local ring of dimension $d$ and residue field $k$. Assume $A$ is not a complete intersection.
Assume \\ $\lim_{n \rt \infty}\beta_{n+1}(k)/\beta_n(k)$ exists (and so equal to $\curv(k)$.
  Let $M$ be an MCM $A$-module with $\curv(M) < \curv(k)$. Then $e(M) \geq \mu(M)(1 + \curv(k))$.
\end{theorem}

We give an example which shows our result is sharp, see \ref{bella}.
\begin{remark}
If $E$ is an $A$-module with $\projdim_A E = \infty$ then in general  we have
$$\limsup_{n \rt \infty} \frac{\beta_{n+1}(E)}{\beta_n(E)} \geq \limsup_{n \rt \infty}\sqrt[n]{\beta_n(E)} \geq  \liminf_{n \rt \infty}\sqrt[n]{\beta_n(E)} \geq \liminf_{n \rt \infty} \frac{\beta_{n+1}(E)}{\beta_n(E)};$$
see \cite[3.37]{RW}.
In all known cases  the limit
$$ \lim_{n \rt \infty} \frac{\beta_{n+1}(E)}{\beta_n(E)} \quad \text{exists and so equal to} \curv(E).$$
 See \cite[4.3.6]{A} for a discussion on this limit.
\end{remark}

\begin{remark}
To the best of our knowledge Theorem \ref{main} is the first result relating multiplicity of a non-extremal MCM $A$-module with $\curv(k)$.
\end{remark}

Next we consider the case when $A$ is a complete intersection.
We prove
\begin{theorem}
\label{main-ci}
Let $(A,\m)$ be a  local complete intersection  of dimension $d$ and residue field $k$.
  Let $M$ be an MCM $A$-module with $\cx(M) < \cx(k)$. Then $e(M) \geq 2\mu(M)$.
\end{theorem}
We now describe in brief the contents of this paper. In section two we discuss a few preliminaries that we need. In section three we describe a construction that we need. In section four we prove Theorem \ref{main} when $\dim A > 0$. In the next section we prove Theorem \ref{main} when $\dim A = 0$.  In section six we give a proof of Theorem \ref{main-ci}. In section seven we give two examples which shows that our results are sharp. In the appendix we compute a limit of a sequence which is crucial for us.
\section{Preliminaries}
In this section we discuss a few preliminaries that we need.
\s \emph{Associated graded rings, modules, Hilbert functions, superficial elements and multiplicity.}\\
Let $(A, \m)$ be local. Let $\gr A = \bigoplus_{n \geq 0}\m^n/\m^{n+1}$ be the \emph{associated graded ring} of $A$.  We note that $\gr A$ is a graded Noetherian $k = A/\m$-algebra.
If $a \in A$ is non-zero then $a \in \m^i \setminus \m^{i+1}$ for some $i$. Then $a^* = $ image of $a$ in $\m^i/\m^{i+1}$ (considered as a subset of $\gr A$) is called the initial form of $a$.
Let $M$ be an $A$-module. Let  $\gr M = \bigoplus_{n \geq 0}\m^nM/\m^{n+1}M$ be the \emph{associated grade module} of $M$. Note $\gr M$ is finitely generated as a $\gr A$-module.

\s The function $H(M, n) = \ell(\m^n M/\m^{n+1} M)$  for $n \geq 0$ is called the Hilbert function of $M$. We assemble it $H_M(z) = \sum_{n \geq 0}H(M, n)z^n$, the Hilbert series of $M$. It is well-known that
\[
H_M(z) = \frac{h_M(z)}{(1-z)^{\dim M}}, \quad \text{where} \ h_M(z) \in \Z[z] \ \text{and} \ h_M(1) \neq 0.
\]
We note  $h_M(1) = e(M)$ the multiplicity of $M$.

\s An element $x \in \m$ is said to be $M$-superficial if there exists $c $ and $n_0$ such that for all $n \geq n_0$ we have $(\m^{n+1}M \colon x)\cap \m^c M = \m^{n}M$.
Superficial elements exist if $k$ is infinite. A sequence $x_1, \ldots, x_s$ is called an $M$-superficial sequence if $x_i$ is $M/(x_1, \ldots, x_{i-1})$-superficial for $i = 1, \ldots, s$.

\s If $\depth M > 0$ then every $M$-superficial element $x$ is $M$-regular. Furthermore in this case we have $(\m^{n+1}M \colon x)  = \m^n M$ for $n \gg 0$.

\s \label{sup-reg} Let $x_1, \ldots, x_r$ be a $M$-superficial sequence. Then $\depth \gr M \geq r$ if and only if $x_1^*, \ldots, x_r^*$ is $\gr M$-regular, see \cite[Theorem 8]{P-th}.

\s\label{mult-sup} Suppose $\depth M \geq r$ and $x_1, \ldots, x_r$ is a $M$-superficial sequence. Set $N = M/(x_1, \ldots, x_r)M$. Then $e(N) = e(M)$, see \cite[Corollary 11]{P-th}.

We need the following elementary fact. We give a proof for the convenience of the reader.
\begin{proposition}
\label{calc} Let $\{ \theta_n \}_{n \geq i_0} $ be positive real numbers such that \\  $\lim_{n \rt \infty} \theta_{n+1}/\theta_{n} = \alpha > 1$. Let $\{r_n\}_{n \geq i_0}$ be a sequence of non-negative real numbers such that $\limsup_{n\rt \infty}\sqrt[n]{r_n} < \alpha$. Then
\[
\lim_{n \rt \infty} \frac{r_n}{\theta_n} = 0.
\]
\end{proposition}
\begin{proof}
Set $c = \limsup_{n\rt \infty}\sqrt[n]{r_n}$.
  By \cite[3.37]{RW} we have $\lim_{n \rt \infty}\sqrt[n]{\theta_n} = \alpha$. Let $\epsilon > 0$ be such that $\alpha - \epsilon  > c + \epsilon$. So for $n \gg 0$
  we have $\theta_n \geq (\alpha-\epsilon)^n$ and $r_n \leq (c+\epsilon)^n$. It follows that
  $$\limsup \sqrt[n]{r_n/\theta_n} \leq \frac{c + \epsilon}{\alpha - \epsilon} < 1. $$
  Therefore the series $\sum_{n \geq i_0} r_n/\theta_n$ is convergent. The result follows.
\end{proof}
\s \label{setup-C} Let $(A,\m)$ be a   local ring.
 Let $D^b(A)$ be the bounded derived category of $A$. We often identify $D^b(A)$ with $K^{-,b}(\proj A)$. We index complexes cohomologically.

\s Let $\Xb \in D^b(A)$. We assume $\Xb \in K^{b,-}(\proj A)$.  Let $$\beta_n(\Xb) = \ell(\Hom_{D^b(A)}(\Xb, k[n])).$$ Note if $\Xb$ is a minimal complex then $\beta_n(\Xb)$ is the rank of the free $A$-module ${\Xb}^{-n}$. Set $$\chi(\Xb)  = \sum_{n \in \Z}(-1)^i\ell(H^i(\Xb)).$$

\s Let $\Yb$ be a  bounded above complex of finitely generated $A$-modules. We do NOT assume $H^i(\Yb) = 0$ for $i \ll 0$. For $m \in \Z$ set
$$\chi_m(\Yb) = \sum_{j \geq 0}(-1)^j\ell(H^{m + j}(\Yb)). $$

\section{A construction }
In this section $(A,\m)$ is a Cohen-Macaulay ring of dimension $d \geq 1$ with infinite residue field. We assume that $A$ is not a complete intersection. So $\curv k > 1$. We also assume that
$$\lim_{n \rt \infty} \beta_{n+1}^A(k)/\beta_n^A(k)  = \curv(k).$$
Set $\alpha = \curv(k)  > 1$.

 \s \label{rc-body} Let $\Pb$ be the minimal projective resolution of $k$ as a $A$-module. Let $g \in \m$ be $A$-regular.  Let $\ov{\Pb} = \Pb\otimes_A A/(g)$.
 We have an exact sequence
 \[
 0 \rt \Pb \xrightarrow{g} \Pb \rt \ov{\Pb} \rt 0.
 \]
 It follows that $H^{n}(\ov{\Pb}) = 0$ for $n \ll 0$. Furthermore
   $$\chi(\ov{\Pb}) = \sum_{n \in \Z}(-1)^i\ell(H^i(\ov{\Pb})) = 0. $$
 Let $M $ be an MCM $A$-module with $\curv(M) < \curv(k)$.  Let $\bx = x_1, \ldots, x_{d}$ be maximal $M \oplus  A$-superficial sequence.
 Set $\Xb = \Pb \otimes_{A} A/(\bx)$. Then one can show similarly $H^{n}(\Xb) = 0$ for $n \ll 0$. Furthermore  $\chi(\Xb) = 0$.
 Set $B = A/(\bx)$. Note $\Xb \in K^{-,b}(\proj B)$ is a minimal complex and $\Hom_{D^b(B)}(\Xb, k[n]) = \beta_{n}$ where $\beta_n = \beta_n^A(k) $ since
 $\beta_n^A(k) = \rank_A {\Pb}^{-n} = \rank_B {\Xb}^{-n} $.  Set $D = M/\bx M$. By \ref{mult-sup} we get $\ell(D) = e(M)$. We also have $\curv_B(D) = \curv_A(M) < \alpha$.

\begin{lemma}\label{rachel}
(with setup as in \ref{rc-body})
 \begin{enumerate}[\rm (a)]
     \item
     $$ \lim_{n \rt \infty} \frac{ |\chi_{-n}(D\otimes_B \Xb)|}{\beta_n^A(k)} = 0.$$
     \item
     $$ \lim_{n \rt \infty} \frac{ \ell(H^{-n}(D\otimes_B \Xb))}{\beta_n^A(k)} = 0.$$
     \item
     Set $\ov{\Xb} = \Xb\otimes_B k$. Then $$\lim_{n \rt \infty} \frac{\chi_{-n}(\ov{\Xb})}{\beta^A_n(k)} = \frac{\alpha}{\alpha + 1}.$$
   \end{enumerate}
\end{lemma}
\begin{proof}
  Recall $\Xb$ is a minimal complex of free $B$-modules. We have $\beta_n = \beta_n^A(k) = \rank_B{\Xb}^{-n}$. 
We assume $H^i(\Xb) = 0$ for $i \leq m_0$. Let $\Omega^i(D)$ denote the $i^{th}$-syzygy of $D$ as a $B$-module.
 Consider
\[
0 \rt \Omega^1(D) \rt B^{\mu(D)} \rt D \rt 0.
\]
Taking $-\otimes \Xb$ we get
\begin{equation*}
  0 \rt \Omega^1(D)\otimes\Xb \rt {\Xb}^{\mu(D)} \rt D\otimes \Xb \rt 0. \tag{$\dagger$}
\end{equation*}
We have $\chi(\Xb) = \sum_{i\in \Z} (-1)^i\ell(H^i(X)) = 0$. Let $m \leq m_0 $ and assume $m + c = m_0$. By ($\dagger$) we have
\[
\chi_m(D\otimes \Xb) = \chi_{m+1}(\Omega^1(D)\otimes \Xb) = \chi_{m+2}(\Omega^2(D)\otimes \Xb) = \cdots = \chi_{m_0}(\Omega^c(D)\otimes \Xb).
\]

(a)  We have $\ell(\Omega^c(D)) \leq \beta^B_{-m + m_0}(D)\ell(B)$. We also have $\beta^B_i(D) = \beta^A_i(M)$ for all $i \geq 0$.
By \ref{calc} it follows that
$$ \frac{|\chi_m(D\otimes \Xb)| }{ \beta_{-m}(k)} \leq \frac{\beta_{-m + m_0}^B(D)\ell(B)}{\beta_{-m}(k)} (\sum_{n \leq -m_0}\beta_n(k)) \rt 0 \quad \text{as} \ -m \rt \infty.$$

(b) We have $\ell(\Omega^c(D)) \leq \beta_{-m + m_0}^B(D)\ell(B)$.
By \ref{calc} it follows that
$$ \frac{\ell(H^m(D\otimes \Xb)) }{ \beta_{-m}(k)} \leq \frac{\beta_{-m + m_0}^B(D)\ell(B)}{\beta_{-m}(k)}\beta_{-m_0}(k) \rt 0 \quad \text{as} \ -m \rt \infty.$$

(c) We note that $\ov{\Xb}$ is a complex with trivial differentials. We have $H_{-m}(\ov{\Xb}) = k^{\beta_m}$ where $\beta_{n} = \beta_n^A(k)$.
It follows that
$$\frac{\chi_{-n}(\ov{\Xb})}{\beta_n} = 1 - \sum_{j \geq 0}(-1)^j\frac{\beta_{n-1-j}}{\beta_n}.$$
As $\lim_{n \rt \infty}\beta_{n+1}/\beta_n = \alpha >1$, it follows from Lemma \ref{lim} that
 $$\lim_{n \rt \infty} \frac{\chi_{-n}(\ov{\Xb})}{\beta_n} = 1 - \frac{1}{\alpha+1} =  \frac{\alpha}{\alpha + 1}.$$
\end{proof}
\section{proof of Theorem \ref{main} when $\dim A > 0$}
We  give
\begin{proof}[Proof of Theorem \ref{main} when $\dim A > 0$]
If the residue field of $A$ is finite then set $A^\prime = A[X]_{\m A[X]}$. The residue field $l$ of $A^\prime$ is $k(X)$ which is infinite. We have the maximal ideal of $A^\prime $ is $\n = \m A^\prime$. It is readily verified that $\beta^A_n(k) = \beta^{A^\prime}_n(l)$. So $\curv_A(k)  = \curv_{A^\prime}(l) = \alpha$. Furthermore the limit $\beta_{n+1}^{A^\prime}(l)/\beta_{n}^{A^\prime}(l)$ exists. Set $N = M\otimes_A A^\prime$. Then $N$ is an MCM $A^\prime$-module. We have $e(N)= e(M)$ and $\mu(N) = \mu(M)$. Furthermore if $\mathbb{F}$ is a minimal free resolution of $M$ as a $A$-module then $\mathbb{F}\otimes_A A^\prime$ is a minimal free resolution of $N$ as an $A^\prime$-module. It follows that $\beta_n^{A^\prime}(N) = \beta_n^A(M)$ for all $n \geq 0$.
So $\curv_{A^\prime}(N) = \curv_A(M) < \alpha$. With these considerations we may assume that the residue field of $A$ is infinite. We do the construction as in section three.

 Throughout $-\otimes - = - \otimes_B -$. We have an exact sequence
$$ 0 \rt \m D \rt D \rt k^{\mu(D)} \rt 0.$$
Let $\Xb$ be as in section three. Taking $-\otimes \Xb$ we get
$$ 0 \rt \Xb \otimes \m D \rt D \otimes \Xb \rt \ov{\Xb}^{\mu(D)} \rt 0$$
where $\ov{\Xb} = \Xb\otimes k$. Set $\beta_n = \beta_n^A(k)$.

Claim-1:
$$ \lim_{n \rt \infty}\frac{\ell(\chi_{-n}(\m D\otimes \Xb))}{ \beta_n} = \mu(D) \frac{\alpha^2}{\alpha + 1}.$$

Proof of Claim-1:  Assume $H^i(\Xb) = 0$ for $i \leq m_0$.
Let $-n \leq m_0$. Then we have exact sequences
\begin{equation}\label{first}
 0 \rt C_n \rt H^{-n}(\m D \otimes X) \rt H^{-n}(D\otimes \Xb) \rt H^{-n}(\ov{\Xb})^{\mu(D)} \rt \cdots.
\end{equation}
and
\begin{equation}\label{second}
  \rt H^{-n -1}(D\otimes \Xb) \rt  H^{-n-1}(\ov{\Xb})^{\mu(D)} \rt C_n \rt 0.
\end{equation}
We have by  \ref{rachel}(b) that
\begin{align*}
   \lim_{n \rt \infty}\frac{\ell(H^{-n -1}(D\otimes \Xb)) }{\beta_n} &= \lim_{n \rt \infty}\frac{\ell(H^{-n -1}(D\otimes \Xb)) }{\beta_{n+1}}  \frac{\beta_{n+1}}{\beta_n} \\
  &=  0 . (\alpha) = 0.
\end{align*}
So by exact sequence (\ref{second}) we get
\begin{equation}\label{third}
  \lim_{n \rt \infty} \frac{\ell(C_n)}{\beta_n} = \mu(D) \lim_{n \rt \infty} \frac{\beta_{n+1}}{\beta_n} = \mu(D) \alpha.
\end{equation}
By the exact sequence (\ref{first}) we obtain
$$ \frac{\ell(C_n)}{\beta_n} + \frac{\ell(\chi_{-n}(D\otimes \Xb))}{ \beta_n}  =  \frac{\ell(\chi_{-n}(\m D\otimes \Xb))}{ \beta_n}  + \mu(D)\frac{\chi_{-n}(\ov{\Xb})}{\beta_n}.  $$
By \ref{rachel}(a),(c) and (\ref{third}) we obtain
$$\alpha \mu(D) = \lim_{n \rt \infty}\frac{\ell(\chi_{-n}(\m D\otimes \Xb))}{ \beta_n}  +  \mu(D) \frac{\alpha}{\alpha + 1}. $$
So we get Claim-1.

Note $\m D \neq 0$. Otherwise $D = k^{\mu(D)}$ and we get $\curv_B(D) = \curv_B(k)$. We have $\curv_B(D) = \curv_A(M)$. By \cite[2.3]{P-ext} we also get $\curv_B(k) = \curv_A(k)$. So we obtain $\curv_A(M) = \curv_A(k)$ which is a contradiction.  So $\m D \neq 0$.

Assume for some $r \geq 1$ we have $\m^{r + 1} D = 0$ and $\m^{r} D \neq 0$.

Case 1: $r \geq 2$. \\
For $j = 1, \ldots, r-1$ we have
exact sequence $ 0 \rt \m^{j+1} D \rt \m^j D \rt \m^j D /\m^{j+1}D \rt 0$. So we obtain exact sequence
\begin{equation}\label{baton}
0 \rt T_n^{(j)} \rt H^{-n}(\m^{j+1} D \otimes X) \rt H^{-n}(\m^j D\otimes \Xb) \rt H^{-n}(\ov{\Xb})^{\mu(\m^j D)} \rt \cdots.
\end{equation}
Note $\m^r M = k^{\mu(\m^r M)}$.
We obtain for $j = 1, \ldots, r-1$
$$ \frac{\ell(T_n^{(j)})}{\beta_n} + \frac{\ell(\chi_{-n}(\m^j D\otimes \Xb))}{ \beta_n}  =  \frac{\ell(\chi_{-n}(\m^{j+1} D\otimes \Xb))}{ \beta_n}  + \mu(\m^j D)\frac{\chi_{-n}(\ov{\Xb})}{\beta_n}. $$
Adding terms we obtain
$$ \frac{\sum_{j = 1}^{r-1}\ell(T_n^{(j)})}{\beta_n}   + \frac{\ell(\chi_{-n}(\m D\otimes \Xb))}{ \beta_n}   = (\ell(D) - \mu(D))\frac{\chi_{-n}(\ov{\Xb})}{\beta_n}. $$
Set $$s_n =  \frac{\sum_{j = 1}^{r-1}\ell(T_n^{(j)})}{\beta_n}.$$
Note $s_n \geq 0$. By \ref{rachel}(c) and Claim-1 we obtain
$$\lim_{n \rt \infty}s_n + \mu(D)\frac{\alpha^2}{\alpha +1} = (\ell(D) - \mu(D))\frac{\alpha}{\alpha +1}.$$
So we obtain
$$\lim_{n \rt \infty}s_n  = \ell(D) \frac{\alpha}{\alpha + 1}  - \mu(D)\alpha. $$
But $s_n \geq 0$. So $\lim_{n\rt \infty}s_n  \geq 0$. Thus
$$\ell(D)  \geq \mu(D)(\alpha + 1).$$

Case 2: $r = 1$. \\
Then $\m D = k^{\mu(\m D)}$. So we have $\chi_{-n}(\Xb \otimes \m D) = \mu(\m D) \chi_{-n}(\ov{\Xb})$. Therefore by \ref{rachel}(c) we obtain
$$\lim_{n \rt \infty}  \frac{\ell(\chi_{-n}(\m D\otimes \Xb))}{ \beta_n}   =  \mu(\m D) \frac{\alpha}{\alpha + 1}.$$
By Claim-1 it follows that $\mu(\m D)/\mu(D) = \alpha$. As $\ell(D) = \mu(D) + \mu(\m D)$ it follows that $\ell(D) = \mu(D)(\alpha + 1).$

Thus in both cases we have $\ell(D)  \geq \mu(D)(\alpha + 1).$
The result follows as $e(M) = \ell(D)$ and $\mu(M) = \mu(D)$.
\end{proof}
\section{Proof of Theorem \ref{main} when $\dim A = 0$}
In this section we give
\begin{proof}[Proof of Theorem \ref{main} when $\dim A = 0$]
Consider the ring $R = A[X]_{(\m, X)}$ which is a flat local extension of $A$ with maximal ideal $\n = (\m, X)R$ and residue field $k$.  Set $N = M\otimes_A R$. We note that $R$ is \CM \ and $N$ is an MCM $R$-module. As $R/(X) = A$ and $A$ is not a complete intersection it follows that $R$ is also not a complete intersection. We have $\gr R = (\gr A)[X^*]$ and $\gr N = (\gr M)[X^*]$. We have $X^*$ is $\gr R \oplus \gr N$-regular. It follows that $e(N) = e(N/XN) = e(M)$ and $\mu(N) = \mu(M)$. If $\mathbb{F}$ is a minimal resolution of $M$ over $A$ then
$\mathbb{F}\otimes_A R$ is a minimal resolution of $N$ as an $R$-module. So $\curv_R(N) = \curv_A(M)$.

Set $\curv_A(k) = \alpha$. By our hypothesis we have $\lim_{n \rt \infty} \beta_{n+1}^A(k)/\beta_n^A(k) = \alpha$.

Set $\theta_n = \beta_n^R(k)$ and $\alpha_n = \beta_n^A(k)$. As $X \notin \n^2$ by \cite[3.3.5]{A},   we get that $\theta_n = \alpha_n + \alpha_{n-1}$ for $n \geq 1$. So we obtain
$$\frac{\theta_{n+1}}{\theta_n} = \frac{\alpha_{n + 1} + \alpha_{n}}{\alpha_n + \alpha_{n-1}} =  \frac{\frac{\alpha_{n+1}}{\alpha_n} +1 }{ 1 + \frac{\alpha_{n-1}}{\alpha_n} }.$$
Taking limits we obtain
$$ \lim_{n \rt \infty} \frac{\theta_{n+1}}{\theta_n} = \frac{\alpha +1}{ 1 + \frac{1}{\alpha}} = \alpha.$$
The \CM \ ring $R$ and the MCM module $N$ satisfies the conditions of Theorem \ref{main} and $\dim R  = 1$. Therefore
$e(N) \geq \mu(N)(1+\alpha)$. The result follows as $e(M) = e(N)$ and $\mu(M) = \mu(N)$.
\end{proof}
\section{Proof of Theorem \ref{main-ci}}
In this section we give proof of Theorem \ref{main-ci}. Throughout this section $A$ is a complete intersection of codimension $c \geq 1$.
We need a few preliminaries.
\s\label{r-ci}(see \cite[9.2.1]{A})
Let $E$ be an $A$-module with infinite projective dimension. Then \\ $\lim_{ n\rt \infty} \beta_n(M)/n^{\cx(E) - 1}$ is positive. Furthermore $\lim_{n \rt \infty} \beta_{n+1}(E)/\beta_n(E) = 1$.
\begin{proposition}\label{ovi}
(with hypotheses as above) Let $E$ be an $A$=module. Then \\$\lim_{n \rt \infty} \beta_n(E)/\beta_n(k)$ exists. Furthermore
$$\lim_{n \rt \infty} \beta_n(E)/\beta_n(k) = \begin{cases}
   & 0 \ \mbox{if } \ \cx(E) < \cx(k) \\
   &>0 \ \mbox{otherwise}.
\end{cases}$$
\end{proposition}
\begin{proof}
We have nothing to show if $E$ has finite projective dimension. So assume $\projdim E = \infty$. Set $\cx E = r$. Note $\cx k = c = \codim A$.
We have
$$\lim_{n \rt \infty} \beta_n(E)/\beta_n(k)=
\frac{\beta_n(E)}{n^{r-1}} \frac{n^{c-1}}{\beta_n(k)}  \frac{n^{r-1}}{n^{c-1}}.$$
The result follows from \ref{r-ci}.
\end{proof}
We now give
\begin{proof}[Proof of Theorem \ref{main-ci}]
If $d = \dim A > 0$ then by an argument similar to proof of Theorem \ref{main} we may assume that the residue field of $A$ is infinite. Let $\bx = x_1,\ldots, x_d$ be a $M \oplus A$-superficial sequence. Set $B = A/(\bx)$ and $D = M/\bx M$. We have $B$ is a complete intersection of codimension $c$. We also have $\cx_B D = \cx_A M < c$. Furthermore $\ell(D) = e(M)$ and $\mu(D) = \mu(M)$.

We have nothing to show if $D$ is a free $B$-module. So assume $D$ is not free. Note $\m D \neq 0$ for otherwise $D \cong k^{\mu(D)}$ and we get $\cx_B D = c$, a contradiction.

Claim-1. $\lim_{n \rt \infty} \beta_n(\m D)/\beta_n(k) = \mu(D)$.

We have an exact sequence  $0 \rt \m D \rt D \rt k^{\mu(D)} \rt 0$.
So we have  $$\mu(D)\beta_n(k) \leq \beta_n(D) + \beta_{n-1}(\m D). $$ By \ref{ovi} it follows that $$\mu(D) \leq \lim_{n \rt \infty} \beta_{n-1}(\m D) /\beta_n(k).$$  As $\lim_{n \rt \infty} \beta_n(\m D)/\beta_{n-1}(\m D) = 1$ it follows that  $\mu(D) \leq \lim_{n \rt \infty} \beta_{n}(\m D) /\beta_n(k)$.  The exact sequence also implies
$$\beta_n(\m D) \leq \mu(D) \beta_{n+1}(k) + \beta_n(D).$$ By \ref{ovi} and \ref{r-ci} we get $\lim_{n \rt \infty} \beta_{n}(\m D) /\beta_n(k)  \leq \mu(D)$. Claim-1 follows.

Assume $\m^r D \neq 0$ and $\m^{r + 1}D = 0$.

Case-1 $r \geq 2$. \\
For $i = 1, 2, \ldots, r-1$ we have exact sequences
$$ 0 \rt \m^{i+1} D \rt \m^i D \rt k^{\mu(\m^i D)} \rt 0.$$
So we get
$$\lim_{n \rt \infty} \frac{\beta_n(\m^i D)}{\beta_n(k)} \leq \lim_{n \rt \infty} \frac{\beta_n(\m^i D)}{\beta_n(k)} + \mu(\m^i D). $$
Adding the terms and as $\m^r D = k^{\mu(\m^r D)}$ we get
$$\lim_{n \rt \infty} \frac{\beta_n(\m D)}{\beta_n(k)} \leq  \sum_{i = 1}^{r} \mu(\m^i(D))  = \ell(D) - \mu(D).$$
By Claim-1 it follows that $\ell(D) - \mu(D) \geq \mu(D)$. So $\ell(D) \geq 2\mu(D)$.

Case-2. $r = 1$.

Then $\m D = k^{\mu (D)}$. So we get $\lim_{n \rt \infty} \beta_{n}(\m D) /\beta_n(k)  =  \mu( \m D)$. By Claim-1 we get $\mu(D) = \mu(\m D)$. It follows that
$\ell(D) = \mu(D) + \mu(\m D) = 2 \mu(D)$.

Thus in any case we have $\ell(D) \geq 2\mu(D)$. The result follows.
\end{proof}
\section{Examples}
In this section we give  examples which shows that our results are sharp.
We first give an example which shows that the bound in Theorem \ref{main} is sharp.
\begin{example}
\label{bella} Let $(R,\m)$ be an one-dimensional \CM \ local ring of minimal multiplicity and infinite residue field. So $e(R) = h + 1$ where $h = \embdim(R) - 1$. We have $\gr R$ is \CM \
 (see \cite[Theorem 2]{S-min}) and
$h_R(z) = 1 + hz$. Let $x\in \m \setminus \m^2$ be such that $x^*$ is $\gr R$-regular. Then the Hilbert series of $R/(x)$ is $1 + hz$. So  $\type(R) = h$. Thus if $ h \geq 2$ then $R$ is not Gorenstein. In particular it is not a complete intersection.

Set $A = R/(x^2)$. Then $\curv_A(k) = h$ and $\lim_{n \rt\infty} \beta_{n+1}(k)/\beta_n(k) = h$, see \ref{min-mult}.
Consider $M = R/(x)$. Then $M$ is an $A$-module. Note we have an exact sequence of $A$-modules
$$0 \rt M \rt A \rt M \rt 0.$$
Thus the first syzygy of $M$ is $M$ itself. So $M$ is a periodic $A$-module. Therefore $\curv(M) = 1 < h$. We have $\ell(M) = \ell(R/(x)) = h + 1$. Also $\mu(M) = 1$. Thus $M$ attains the bound of Theorem \ref{main}.
\end{example}
\begin{remark}
If $\curv(k)$ is irrational then clearly $e(M) \neq \mu(M)(1+\curv(k))$. So the bound is not attained in this case.
\end{remark}

\s The following result is well-known. We give a proof as we don't have a reference.
\begin{lemma}\label{min-mult}
Let $(S,\n)$, $(T,\m)$ be rings with minimal multiplicity of codimension $h \geq 2$ and dimensions $0, 1$ respectively. Let $f \in \m^2$ be a non-zero divisor on $T$ and let $A = T/(f)$.
Then
\begin{enumerate}[\rm (1)]
  \item $\beta_n^S(k) = h^n$ for all $n \geq 0$.
  \item $P_S(z) = 1/(1-hz)$
  \item $P_T(z) = (1+z)P_S(z)$.
  \item $P_A(z)  =  P_T(z)/(1-z^2) = \frac{1}{(1-hz)(1-z)}$.
  \item $\lim_{n \rt \infty} \beta_{n+1}^A(k)/\beta_n^A(k)  = h$.
\end{enumerate}
\end{lemma}
\begin{proof}
(1) Note $\n = k^h$. It follows that $\beta_1(k) = h$ and $\beta_{n + 1} = h \beta_n$ for $n \geq 1$. The result follows.

(2) This follows from (1).

(3) After possibly going to a flat extension (to get infinite residue field) we may assume that there exists $x \in \m$ which is $T$-superficial. So $T/(x)$ has minimal multiplicity. The result follows from \cite[3.3.5]{A}.

(4) This follows from \cite[3.3.5]{A}.

(5) We have by (2), (3)
\[
\sum_{n \geq 0} \beta^A_n(k)z^n  = \frac{P_S(z)}{(1-z)}.
\]
Set $\beta_n = \beta^A_n(k)$. Then by above equality we obtain
\[
\beta_n = 1 + h + \cdots + h^n = \frac{h^{n+1} -1}{h -1}.
\]
So we obtain
\[
\frac{\beta_{n+1}}{\beta_n} = \frac{h^{n+1} -1}{h^{n} -1} = \frac{ h - \frac{1}{h^n} }{1 - \frac{1}{h^n}}.
\]
Taking limits as $n \rt \infty$ we obtain the result.
\end{proof}

Next we give an example which shows that the bound in Theorem \ref{main-ci} is sharp.
\begin{example}
\label{ex-ci} Let $R = k[[X, Y]]/(X^2)$, $A = R/(Y^2)$ and $M = R/(X)$. As before $M$ is a periodic $A$-module. So $\cx_A M = 1$. We have $\codim A = 2$ and so $\cx_A k = 2$. We also have $\ell(M) = 2$ and $\mu(M) = 1$. So $M$ attains the bound of Theorem \ref{main-ci}.
\end{example}
\section{Appendix}
In the appendix we calculate a limit which is crucial to us.
\s \label{lim-setup} Let $\{ \theta_n \}_{n \geq c} $ be a sequence of positive real numbers such that $$\lim_{n \rt \infty} \theta_{n + 1}/\theta_n = \xi > 1.$$
 Set $\theta _n = 0$ for $n < c$. Also set
$$r_n  = \sum_{j \geq 0}(-1)^j\frac{\theta_{n-1-j}}{\theta_n}. $$
We show
\begin{lemma}\label{lim}
  (with hypotheses as in \ref{lim-setup}). We have
  $$ \lim_{n \rt \infty} r_n  = \frac{1}{\xi + 1}. $$
\end{lemma}
\begin{proof}
We note $\lim_{n \rt \infty} \theta_n = \infty$.
  Choose $\epsilon > 0$ such that $0 < \epsilon < 1/\xi$ and $(1/\xi) + \epsilon < 1$. We note that  $\lim_{n \rt \infty}\theta_n/\theta_{n+1} = 1/\xi$. Choose
  $m_0$ such that
  $$ \frac{1}{\xi} - \epsilon < \frac{\theta_{m-1}}{\theta_m} < \frac{1}{\xi} + \epsilon \quad \text{for all}\  m \geq i_0.$$
  We note for $j \geq 2$  and $m-j \geq i_0$ we have
  $$ \frac{\theta_{m-j}}{\theta_m} = (\frac{\theta_{m-j}}{\theta_{m-j +1}}) (\frac{\theta_{m-j + 1}}{\theta_{m-j+2}} )\cdots ( \frac{\theta_{m-1}}{\theta_m}).$$
  So we have
  \begin{equation*}
    (\frac{1}{\xi} - \epsilon )^j < \frac{\theta_{m-j}}{\theta_m} < (\frac{1}{\xi} + \epsilon )^j. \tag{$\dagger$}
  \end{equation*}
  Let
  $$t_n = \sum_{n-1-j \leq i_0 -1}(-1)^j\frac{\theta_{n-1-j}}{\theta_n}.$$
  Then clearly $\lim_{n \rt \infty} t_n = 0$.
  Set
  $$ \alpha  = (\epsilon + 1/\xi) \quad \text{and} \quad \beta = ((1/\xi) - \epsilon). $$
  We have
  \begin{align*}
    r_n &= t_n + \sum_{n-1-j \geq i_0 }(-1)^j\frac{\theta_{n-1-j}}{\theta_n} \\
     &=  t_n + \sum_{n-1-j \geq i_0, j \ \text{even}}\frac{\theta_{n-1-j}}{\theta_n} \ - \ \sum_{n-1-j \geq i_0, j \ \text{odd} }\frac{\theta_{n-1-j}}{\theta_n} \\
     &\leq t_n + \alpha(1 + \alpha^2 + \cdots + \alpha^{2l}) - \beta^2(1+ \beta^2 + \cdots + \beta^{2s})
  \end{align*}
  where $l, s \rt \infty$ as $n \rt \infty$.
  So we have $r_n \leq t_n + u_n(\epsilon)$ where
  $$u_n(\epsilon) = \alpha \frac{1 - \alpha^{2l +2}}{1 - \alpha^2} - \beta^2 \frac{1-\beta^{2s + 2}}{1 - \beta^2}. $$
  We have that $ 0< \max \{\alpha, \beta \} < 1$. So
  $$u(\epsilon) = \lim_{n \rt \infty} u_n(\epsilon) =  \frac{\alpha}{1 - \alpha^2} - \frac{\beta^2}{1 - \beta^2}.$$
  We have
  $\limsup r_n \leq u(\epsilon)$. Taking $\epsilon \rt 0$ we obtain
  $$ \limsup r_n \leq \frac{1/\xi}{1 - (1/\xi)^2} - \frac{1/\xi^2}{1 -(1/\xi)^2} = \frac{1}{\xi + 1}.$$

We now compute $\liminf r_n$. We note as before
\begin{align*}
  r_n &= t_n + \sum_{n-1-j \geq i_0, j \ \text{even}}\frac{\theta_{n-1-j}}{\theta_n} \ - \ \sum_{n-1-j \geq i_0, j \ \text{odd} }\frac{\theta_{n-1-j}}{\theta_n} \\ \\
   &\geq t_n + \beta(1+ \beta^2 + \cdots + \beta^{2s}) - \alpha^2(1 + \alpha^2 + \cdots + \alpha^{2l})
\end{align*}
where $l, s \rt \infty$ as $n \rt \infty$.
  So we have $r_n \geq t_n + v_n(\epsilon)$ where
   $$v_n(\epsilon) = \beta \frac{1 - \beta^{2s +2}}{1 - \beta^2} - \alpha^2 \frac{1-\alpha^{2s + 2}}{1 - \alpha^2}. $$
  We have that $ 0< \max\{\alpha, \beta \} < 1$. So
   $$v(\epsilon) = \lim_{n \rt \infty} v_n(\epsilon) =  \frac{\beta}{1 - \beta^2} - \frac{\alpha^2}{1 - \alpha^2}.$$
  We have
  $\liminf r_n \geq v(\epsilon)$. Taking $\epsilon \rt 0$ we obtain
  $$ \liminf r_n \geq  \frac{1/\xi}{1 - (1/\xi)^2} - \frac{1/\xi^2}{1 -(1/\xi)^2} = \frac{1}{\xi + 1}.$$
 The result follows.
\end{proof}



\begin{thebibliography} {99}



\bibitem{A}
L.~L.~Avramov,
\emph{Infinite free resolutions},
Six lectures on commutative algebra (Bellaterra, 1996), 1-118, Progr. Math., 166. Birkhäuser, Basel (1998).







\bibitem{P-th}
T.~J.~Puthenpurakal,
\emph{Hilbert-coefficients of a Cohen-Macaulay module},
J. Algebra, 264, (2003),  82--97.

\bibitem{P-ext}
\bysame,
\emph{Obstructions to curvature of modules over Cohen-Macaulay rings},
Preprint, arXiv:2511.16109.

\bibitem{RW}
W.~Rudin,
\emph{Principles of mathematical analysis},
Third edition
International Series in Pure and Applied Mathematics,
McGraw-Hill Book Co., New York-Auckland-D\"{u}sseldorf, 1976.

\bibitem{S-min}
J.~D.~Sally,
\emph{On the associated graded ring of a local Cohen-Macaulay ring},
J. Math. Kyoto Univ., 17, (1977),  19--21.



\end{thebibliography}
\end{document}